\documentclass[12pt,a4paper,fleqn]{article}
\usepackage[T1]{fontenc}
\usepackage[english]{babel}
\usepackage{mathrsfs}
\usepackage{amssymb,amsmath,amsthm,amsfonts}
\usepackage[left=1.5cm,right=1.5cm,top=1.5cm,bottom=1.5cm]{geometry} 
\usepackage{indentfirst}
\usepackage{verbatim}
\usepackage{float}

\usepackage{pgf,tikz}
\usetikzlibrary{arrows}
\usepackage{caption}

\numberwithin{equation}{section}

\usepackage{color}

\vfuzz \hfuzz
\newlength{\saveparindent}
\raggedright
\newtheorem{theorem}{Theorem}[section]
\newtheorem{lemma}[theorem]{Lemma}
\newtheorem{proposition}[theorem]{Proposition}

\newtheorem{corollary}[theorem]{Corollary}

\let\originalleft\left
\let\originalright\right
\renewcommand{\left}{\mathopen{}\mathclose\bgroup\originalleft}
\renewcommand{\right}{\aftergroup\egroup\originalright}

\newcommand{\llav}[1]{  \left\{#1\right\} }
\newcommand{\pic}[1]{  \left\langle #1\right\rangle }
\newcommand{\conv}[1]{\xrightarrow[  #1  \hspace{0.12cm}]{}}
\newcommand{\norm}[1]{  \left\|#1\right\| }
\newcommand{\pare}[1]{\left(#1\right)}
\newcommand{\corch}[1]{  \left[#1\right] }
\newcommand{\abs}[1]{  \left|#1\right| }
\newcommand{\CAL}[1]{\mathcal{#1}}  
\newcommand{\BB}[1]{\mathbb{#1}}

\newcommand{\dis}[1]{\displaystyle{#1}}
\newcommand{\tn}[1]{\textnormal{#1}}

\newcommand{\ul}[1]{\underline{#1}}

\def\mathcolor#1#{\@mathcolor{#1}}
\def\@mathcolor#1#2#3{%
  \protect\leavevmode
  \begingroup
    \color#1{#2}#3%
  \endgroup
}

\newcommand{\blanco}[1]{\mathcolor{white}{#1}}

\def\11{\mbox{\bfseries{1}} }
\def\defi{\mathrel{\mathop:}=}

\def\11{\textnormal{\textbf{1}}}

\DeclareMathOperator{\dr}{\textit{dr}}

\title{Turing instability in a model with two interacting Ising lines: hydrodynamic limit}
\author{Monia Capanna\thanks{Universit\`a degli Studi dell'Aquila, Via Vetoio, 67100 L'Aquila, Italy. Email:\,{\tt monia.capanna@graduate.univaq.it}} ,
Nahuel Soprano-Loto\thanks{Gran Sasso Science Institute, Viale F. Crispi 7, 67100 L'Aquila, Italy. Email:\,{\tt sopranoloto@gmail.com}}}

\begin{document}

\sloppy
\maketitle
 
\begin{abstract}
\noindent 
This is the first of two articles on the study of a particle system model that exhibits  a Turing instability type effect.
The model is based on two discrete lines (or toruses) with Ising spins, 
that evolve according to a continuous time Markov process defined in terms of
macroscopic Kac potentials and local interactions.
For fixed time, we prove that the density fields weakly converge to the solution of a system of partial differential equations involving convolutions.
The presence of local interactions results in the lack of propagation of chaos, reason why the hydrodynamic limit cannot be obtained from previous results.
\end{abstract}

\section{Introduction}

In \cite{Tur37}, it is proposed a reaction-diffusion model for the kinetics of certain chemical substances in order to explain the mechanism underlying the pattern formation in various systems; we briefly describe the part of its content that is relevant for our purpose.
Consider a system of ODE's
\begin{align}\label{55}
\begin{aligned}
\frac{d}{dt}u_1&=f_1\pare{u_1,u_2}
\\[5pt]
\frac{d}{dt}u_2&=f_2\pare{u_1,u_2}
\end{aligned}
\end{align}
for which $\pare{0,0}$ is a stationary, linearly stable solution.
Consider the reaction-diffusion equation
\begin{align}
\begin{aligned}\label{464}
\partial_t u_1&=f_1\pare{u_1,u_2}+d_1 \Delta u_1
\\[5pt]
\partial_t u_2&=f_2\pare{u_1,u_2}+d_2 \Delta u_2
\end{aligned}
\end{align}
associated to the reaction system \eqref{55}.
The linear stability of this system is studied in terms of the Fourier transforms of its linearized version around $\pare{0,0}$.
As diffusions do not affect the behavior of the evolution while considering homogeneous initial conditions, the linear stability of the zero-Fourier transform is inherited from the one of system \eqref{55}.
Under certain hypotheses over $f_1$ and $f_2$, Turing proved that the diffusion coefficients $d_1$ and $d_2$ can be chosen in such a way that
system \eqref{464} exhibits linear instability in some nonzero-Fourier modes.
In other words,
for the system \eqref{464},
the pair function $\pare{\ul 0,\ul 0}$ is stable under homogeneous perturbations but unstable under inhomogeneous ones, and this happens due to the diffusion terms.
This phenomena, known as Turing instability, is anti-intuitive in the sense that the presence of diffusion is understood to have a smoothing effect.
See \cite{Mur03} for an extensive exposition of the theme.

This is the first of two articles devoted to 
introducing a different framework where
Turing instability occurs,
namely a particle system model whose motivation comes from statistical mechanics.
On the one hand, the model converges to a system of PDE's that reproduces the phenomena described in the previous paragraph with an important difference:
in our case, we have non-local mixing parts in the sense that
diffusions are replaced by convolutions by some smoothing functions.
On the other hand,
the Turing effect is also observed microscopically in the sense that pattern formation occurs if we consider the proper space-time scaling.
The microscopic  model consists of two discrete lines (or toruses) of Ising spins.
Each line of spins evolves according to a spin-flip dynamic for which the Gibbs measure associated to a Hamiltonian with ferromagnetic long range interactions is reversible.
More precisely, these interactions respond to macroscopic Kac potentials ---that play the role of the mentioned smoothing functions in the hydrodynamic limit.
We consider different inverse temperatures in each line.
In addition, there are local activating-inhibiting interactions between the two lines in the  sense
that the first line acts as an external field with intensity $\lambda$ over the second one, and vice versa with intensity $-\lambda$.
A side comment:
a first approach to microscopically reproduce Turing's instability would have been to define a model with a Glauber plus Kawasaki dynamic as in \cite{DMFL86} that approximates equations \eqref{55};
we found the approach presented here more realistic from the point of view of physics.

This first paper is devoted to the hydrodynamic limit of the model.
The local interactions represent an important obstacle from the technical point of view because they result in the lack of propagation of chaos.
For this reason, despite the presence of Kac potentials, we cannot conclude from the ideas developed in \cite{DOPT94}, for instance.
The proof relies on an (almost) closed formula obtained for the generator and the adaptation of the techniques exposed in chapter 4 of \cite{KL99} to our model.
The reaction version of the limit equations, obtained by removing the convolution terms and the spatial dependence, coincides with the one obtained by considering mean-field instead of Kac interactions.

In the second article \cite{CSL17}, we study linear stability around the equilibrium point   in the case in which the Kac potentials are Gaussian.
In this framework, we find conditions over the macroscopic parameters under which the zero-Fourier mode is stable and, at the same time, there are nonzero unstable ones.
Under these conditions, we prove that, for a time that converges to the critical time in which the process starts to be finite,
pattern formation occurs in the sense that, despite we start with a translation invariant initial condition, there are nonzero-Fourier modes that do not vanish.

\section{Definitions and statements of the results}

\subsection{The microscopic model}

Consider the unit (macroscopic) torus $\BB T$, that we identify with the real interval $\left[0,1\right)$.
We endow $\BB T$ with the periodic metric $d_{\BB T}$ defined as
\begin{align}\nonumber
d_{\BB T}\pare{r,\tilde r}\defi \min_{a\in\BB Z}\abs{r-\tilde r+a}.
\end{align}
The notions of continuity and measurability associated to $\BB T$ always refer to $d_{\BB T}$: we consider its associated topology, and the $\sigma$-algebra generated by this topology (the Borel $\sigma$-algebra).
The microscopic torus $\Lambda_\gamma$ is defined as $\pare{\gamma^{-1}\BB T}\cap \BB Z$,  $\gamma$ of the form $1/N$, $N\in\BB N$.
The limit $\gamma\to 0$ always refers to $N\to\infty$.
Elements of $\Lambda_\gamma$ are denoted by the letters $x$ and $y$.
For every $\gamma$, we define a continuous time Markov process 
$\pare{\ul\sigma_\gamma\pare{t}}_{t\ge 0}=
\pare{\pare{\sigma_{\gamma,1}\pare{t},\sigma_{\gamma,2}\pare{t}}}_{t\ge 0}$ with state space
$\llav{-1,1}^{\Lambda_\gamma}\times\llav{-1,1}^{\Lambda_\gamma}$.
All these processes are defined in the same abstract probability space $\pare{\Omega,\CAL F,\BB P}$.
For $t\ge 0$ and $i\in\llav{1,2}$, $\sigma_{\gamma,i}\pare{t,x}\in\llav{-1,1}$ denotes the spin of $\sigma_{\gamma,i}\pare{t}$ at the site $x$.
The initial distribution is given in terms of a pair of functions $\psi_1,\psi_2\in C\pare{\BB T,\corch{-1,1}}$ as follows:
the family $\llav{\sigma_{\gamma,1}\pare{0,x}:x\in\Lambda_\gamma}\cup
\llav{\sigma_{\gamma,2}\pare{0,x}:x\in\Lambda_\gamma}$ is independent, and
$\BB E\pare{\sigma_{\gamma,i}\pare{0,x}}=\psi_i\pare{\gamma x}$ for every $i\in\llav{1,2}$ and  $x\in\Lambda_\gamma$.
Before defining the generator of our Markov process, we need to introduce two kernels $\phi_1,\phi_2\in C\pare{\BB T\times \BB T,[0,\infty)}$ that define the interactions between the spins ($\phi_i$ defines the interactions between the spins in the $i$-th line, $i\in\llav{1,2}$).
We ask for them to be symmetric ($\phi_i\pare{r,\tilde r}=\phi_i\pare{\tilde r, r}$ for every $\pare{r,\tilde r}\in\BB T\times \BB T$),
translation invariant ($\phi_i\pare{r,\tilde r}=\phi_i\pare{0,d_{\BB T}\pare{r,\tilde r}}$  for every $\pare{r,\tilde r}\in\BB T\times \BB T$),
and to integrate $1$ ($\int_0^1\phi_i\pare{0,r}\dr=1$).
For $\sigma_\gamma\in\llav{-1,1}^{\Lambda_\gamma}$ and $i\in\llav{1,2}$, we define the discrete convolution
\begin{align}\nonumber
\pare{\sigma_\gamma *\phi_i}\pare{x}\defi \gamma\sum_{y\in\Lambda_\gamma}\sigma_\gamma\pare{y}\phi_i\pare{\gamma x,\gamma y}.
\end{align}
We then define the ferromagnetic Hamiltonian
\begin{align}\label{hamiltonian}
H_i\pare{\sigma_\gamma}\defi -\frac{1}{2}\sum_{x\in\Lambda_\gamma}\sigma_\gamma\pare{x}\pare{\sigma_\gamma*\phi_i}\pare{x}, \ \sigma_\gamma\in\llav{-1,1}^{\Lambda_\gamma}.
\end{align}
In each of the two lines of the state space $\llav{-1,1}^{\Lambda_\gamma}\times \llav{-1,1}^{\Lambda_\gamma}$, we have an associated inverse temperature $\beta_i>0$, $i\in\llav{1,2}$.
We also have a parameter $\lambda>0$ that describes the interaction between the two lines.
The generator of the Markov process is of the spin-flip type.
For a pair configuration $\ul\sigma_\gamma=\pare{\sigma_{\gamma,1},\sigma_{\gamma,2}}\in \llav{-1,1}^{\Lambda_\gamma}\times \llav{-1,1}^{\Lambda_\gamma}$, the rate of flipping the spin at $x$ in the first line is given by
\begin{align}\nonumber
R_{1}\pare{x,\ul\sigma_\gamma}
=\frac{\exp\llav{-\beta_1\sigma_{\gamma,1}(x)
\pare{\sigma_{\gamma,1}*\phi_1}\pare{x}}
\exp\llav{-\beta_1\lambda\sigma_{\gamma,1}(x)
 \sigma_{\gamma,2}(x)}
}
{2\cosh  \llav{
\beta_1\pare{\sigma_{\gamma,1}*\phi_1}\pare{x}+\beta_1\lambda\sigma_{\gamma,2}\pare{x}}},
\end{align}
while the rate of flipping the spin at $x$ in the second line is given by
\begin{align}\nonumber
R_{2}\pare{x,\ul\sigma_{\gamma}}
=\frac{\exp\llav{-\beta_2\sigma_{\gamma,2}(x)
\pare{\sigma_{\gamma,2}*\phi_2}\pare{x}}
\exp\llav{\beta_2\lambda\sigma_{\gamma,2}(x)
 \sigma_{\gamma,1}(x)}
}
{2\cosh  \llav{
\beta_2\pare{\sigma_{\gamma,2}*\phi_2}\pare{x}-\beta_2\lambda\sigma_{\gamma,1}\pare{x}}}.
\end{align}
In these rates, the denominator are mobility coefficients.
In the numerators, we have one part that is the reversible dynamic associated to the Hamiltonian \eqref{hamiltonian} with inverse temperature $\beta_i$, $i\in\llav{1,2}$.
At the same time, the lines are correlated in an attractive-repulsive way as follows: $\sigma_{\gamma,1}(t,x)$ wants to coincide with $\sigma_{\gamma,2}(t,x)$ with a rate that depends on $\beta_1\lambda$, while $\sigma_{\gamma,2}(t,x)$ wants to differ from $\sigma_{\gamma,1}(t,x)$ with a rate that depends on $\beta_2\lambda$.
The choice of the sign of $\lambda$ is not a restriction: the case $\lambda=0$ makes the lines independent and the problem not interesting; the case $\lambda<0$ only interchanges the roles of the lines.


\subsection{Hydrodynamic limit}\label{abc}

For every $\gamma$, the correlation field is the random process $\pare{\eta_{\gamma}\pare{t}}_{t\ge 0}$, taking values in $\llav{-1,1}^{\Lambda_\gamma}$, defined as
$\eta_{\gamma}\pare{t,x}\defi\sigma_{\gamma,1}\pare{t,x}\sigma_{\gamma,2}\pare{t,x}$, $x\in\Lambda_\gamma$.
The solutions of the system of partial differential equations
\begin{align}
\begin{aligned}\label{hydro1}
\partial_t u_1\pare{t,r}=&-u_1 
\\[5pt] &+\frac{1}{2}\corch{\tanh\pare{\beta_1 u_1 * \phi_{1}+\beta_1\lambda}
+\tanh\pare{\beta_1 u_1 * \phi_{1}-\beta_1\lambda}} 
\\[5pt] & +u_2\frac{1}{2}\corch{\tanh\pare{\beta_1 u_1 * \phi_{1}+\beta_1\lambda}
-\tanh\pare{\beta_1 u_1 * \phi_{1}-\beta_1\lambda}}
\end{aligned}
\end{align}
\begin{align}
\begin{aligned}\label{hydro2}
\partial_t u_2\pare{t,r}=&-u_2 
\\[5pt] &+\frac{1}{2}\corch{\tanh\pare{\beta_2 u_2 * \phi_{2}+\beta_2\lambda}
+\tanh\pare{\beta_2 u_2 * \phi_{2}-\beta_2\lambda}} 
\\[5pt] 
& -u_1\frac{1}{2}\corch{\tanh\pare{\beta_2 u_2 * \phi_{2}+\beta_2\lambda}
-\tanh\pare{\beta_2 u_2 * \phi_{2}-\beta_2\lambda}}
\end{aligned}
\end{align}
\begin{align}
\begin{aligned}\label{hydro3}
\partial_t v\pare{t,r}=&-2v
\\[5pt]  &+\frac{1}{2}\corch{\tanh\pare{\beta_1 u_1 * \phi_{1}+\beta_1\lambda}
-\tanh\pare{\beta_1 u_1 * \phi_{1}-\beta_1\lambda}}
\\[5pt] 
&-\frac{1}{2}\corch{\tanh\pare{\beta_2u_2 * \phi_{2}+\beta_2\lambda}
+\tanh\pare{\beta_2 u_2 * \phi_{2}-\beta_2\lambda}}
\\[5pt]
& +\frac{u_1}{2}\corch{\tanh\pare{\beta_2 u_2 * \phi_{2}+\beta_2\lambda}
+\tanh\pare{\beta_2 u_2 * \phi_{2}-\beta_2\lambda}}
\\[5pt] 
 & +\frac{u_2}{2}\corch{\tanh\pare{\beta_1 u_1 * \phi_{1}+\beta_1\lambda}
+\tanh\pare{\beta_1 u_1 * \phi_{1}-\beta_1\lambda}}
\end{aligned}
\end{align}
are respectively the limits of $\sigma_{\gamma,1}\pare{t}$, $\sigma_{\gamma,2}\pare{t}$ and $\eta_\gamma\pare{t}$.
The precise statement is the content of the following theorem.

\begin{theorem}\label{hydrolimit}
Fix $T>0$ and $G\in C\pare{\BB T,\BB R}$.
The limits
\begin{align}\nonumber
\begin{aligned}
&\sup_{0\le t\le T}\abs{\pic{\sigma_{\gamma,i}\pare{t},G} -\pic{u_i\pare{t,\cdot},G}}\conv{\gamma\to 0}0, \ i\in\llav{1,2}
\\[5pt]
&\sup_{0\le t\le T}\abs{\pic{\eta_{\gamma}\pare{t},G}-\pic{v\pare{t,\cdot},G}}\conv{\gamma\to 0}0
\end{aligned}
\end{align}
hold in $\BB P$-probability,
where $u_1,u_2,v:[0,\infty)\times \BB T \to \BB R$ are the solutions of the system of partial differential equations \tn{(\ref{hydro1}-\ref{hydro3})}
with initial condition $u_i\pare{0,\cdot}=\psi_i$, $i\in\llav{1,2}$, and $v\pare{0,\cdot}=\psi_1\pare{\cdot}\psi_2\pare{\cdot}$.
\end{theorem}

The function $G$ has to be understood as a test function, and the inner products are defined as
\begin{align}\nonumber
\pic{\sigma_\gamma,G}\defi\gamma\sum_{x\in\Lambda_\gamma}\sigma_\gamma\pare{x}G\pare{\gamma x} \ \ \tn{and} \ \
\pic{\psi,G}\defi
\int_0^1 \psi\pare{r}G\pare{r}dr
\end{align}
for $\sigma_\gamma\in\llav{-1,1}^{\Lambda_\gamma}$ and $\psi\in L^1\pare{\BB T,\BB C}$.
We are  also using the convolution notation 
\begin{align}\nonumber
\psi *\phi_i\pare{r}=\int_0^1  \psi\pare{\tilde r}\phi_i\pare{\tilde r,r}\,d\tilde r.
\end{align}
As it is standard that the system of PDE's in the previous theorem has a unique solution (because of the Lipschitzianity of its data), we will skip the proof.
Observe that, despite the initial condition for $v$ is the product of the initial conditions for $u_1$ and $u_2$, this is not the case for strictly positive times; in other words, propagation of chaos does not occur.

By considering $\phi_1\pare{0,\cdot}=\phi_2\pare{0,\cdot}\equiv 1$, we obtain the mean-field case in which, in each line, the spins interact with all the others with the same intensity.
The main observable in this case is the total magnetization $m_{\gamma,i}\pare{t}\defi\dis{\gamma\sum_{x\in\Lambda_\gamma}\sigma_{\gamma,i}\pare{t,x}}=\pic{\sigma_{\gamma,i}\pare{t},\ul 1}$.
The following corollary is immediate from theorem \ref{hydrolimit}. 

\begin{corollary}
For every $T>0$, the limits
\begin{align}\nonumber
\sup_{0\le t\le T}\abs{m_{\gamma,i}\pare{t,\cdot} -m_i\pare{t}}\conv{\gamma\to 0}0, \ i\in\llav{1,2}
\end{align}
hold in $\BB P$-probability, where $m_1$ and $m_2$ are the solutions of the system of differential equations
\begin{align}
\begin{aligned}\label{mfeq1}
\frac{d}{dt}m_1\pare{t}=&-m_1 
\\[5pt] &+\frac{1}{2}\corch{\tanh\pare{\beta_1 m_1 +\beta_1\lambda}
+\tanh\pare{\beta_1 m_1 -\beta_1\lambda}} 
\\[5pt] & +m_2\frac{1}{2}\corch{\tanh\pare{\beta_1 m_1 +\beta_1\lambda}
-\tanh\pare{\beta_1 m_1 -\beta_1\lambda}}
\end{aligned}
\end{align}
\begin{align}
\begin{aligned}\label{mfeq2}
\frac{d}{dt}m_2\pare{t}=&-m_2 
\\[5pt] &+\frac{1}{2}\corch{\tanh\pare{\beta_2 m_2+\beta_2\lambda}
+\tanh\pare{\beta_2 m_2 -\beta_2\lambda}} 
\\[5pt] 
& -m_1\frac{1}{2}\corch{\tanh\pare{\beta_2 m_2 +\beta_2\lambda}
-\tanh\pare{\beta_2 m_2-\beta_2\lambda}}.
\end{aligned}
\end{align}
\end{corollary}


Observe that $\pare{0,0}$ is an equilibrium point of system (\ref{mfeq1}-\ref{mfeq2}).
Observe also that system (\ref{hydro1}-\ref{hydro2}) is obtained from this one by adding convolutions. For this reason, in analogy with the standard terminology used for reaction-diffusion equations,  system (\ref{mfeq1}-\ref{mfeq2}) can be interpreted as the reaction part of (\ref{hydro1}-\ref{hydro2}) if we substitute diffusions by convolutions.
This explains the analogy with the original case presented by Turing:
systems (\ref{mfeq1}-\ref{mfeq2}) and (\ref{hydro1}-\ref{hydro2}) respectively are our versions of systems \eqref{55} and \eqref{464}.
In \cite{CSL17}, we study conditions under which the linearized version of system (\ref{hydro1}-\ref{hydro2}) has stable zero-Fourier mode 
(or, equivalently, the linearized version of system (\ref{mfeq1}-\ref{mfeq2}) is stable)
but some unstable nonzero ones.

\section{Proof of Theorem \ref{hydrolimit}}

The main point is an almost closed formula for the generator applied to the density and correlations fields, whose obtainment requires the use of a trick that is specific to Ising spins and that has been used before in \cite{DOPT94}.
Once this formula has been obtained, we are able to adapt the techniques developed in chapter 4 of \cite{KL99} to our case:
we first control the martingale terms,
we next prove tightness for the distributions of the process,
and we finally characterize and prove uniqueness of limit point.

\subsubsection*{Almost closed formula for the generator}

Let $E_\gamma\defi\llav{-1,1}^{\Lambda_\gamma}\times \llav{-1,1}^{\Lambda_\gamma}$ be the state-space of the Markov process, and let $L_\gamma:L^\infty\pare{E_\gamma,\BB R}\to L^\infty\pare{E_\gamma,\BB R}$ be its generator:
\begin{align}\nonumber
\begin{aligned}\nonumber
&\pare{L_\gamma F}\pare{\ul\sigma_\gamma} 
\\[5pt]
& \quad=\sum_{x\in\Lambda_\gamma}\llav{R_1\pare{\ul \sigma_\gamma,x}\corch{F\pare{\sigma_{\gamma,1}^x,\sigma_{\gamma,2}}-F\pare{\sigma_{\gamma,1},\sigma_{\gamma,2}}}
+R_2\pare{\ul \sigma_\gamma,x}\corch{F\pare{\sigma_{\gamma,1},\sigma_{\gamma,2}^x}-F\pare{\sigma_{\gamma,1},\sigma_{\gamma,2}}}}.
\end{aligned}
\end{align}
Here $\sigma_{\gamma,i}^x$ is obtained from $\sigma_{\gamma,i}$ by flipping the value of the spin at the site $x$.
For $G\in C\pare{\BB T,\BB R}$ and $i\in\llav{1,2}$, we use the notation $L_\gamma\pic{\sigma_{\gamma,i},G}$ to denote $\pare{L_\gamma F_G}\pare{\ul \sigma_\gamma}$ for $F_G$ the map defined as $F_G\pare{\ul\sigma_{\gamma}}= \pic{\sigma_{\gamma,1},G}$.
Observe that, for every $\ul\sigma_\gamma\in E_\gamma$,
\begin{align}\nonumber
\begin{aligned}
L_\gamma\pic{\sigma_{\gamma,1}, G}&=\sum_{x\in\Lambda_\gamma}R_{1}\pare{x,\ul\sigma_{\gamma}}\pare{\pic{\sigma_{\gamma,1}^x, G}-\pic{\sigma_{\gamma,1}, G}}\\[5pt]
&=-2\gamma\sum_{x\in\Lambda_\gamma}R_{1}\pare{x,\ul\sigma_{\gamma}}\sigma_{\gamma,1}(x)G(\gamma x)\\[5pt]
&=-\pic{\sigma_{\gamma,1}, G}+\gamma\sum_{x\in\Lambda_\gamma}\tanh\corch{ \beta_1  \pare{\sigma_{\gamma,1}*\phi_1}\pare{x}+\beta_1\lambda\sigma_{\gamma,2}\pare{x}}G(\gamma x);
\end{aligned}
\end{align}
the last step follows after writing explicitly the rate $R_1\pare{x, \ul\sigma_{\gamma}}$ and using identity $\sigma_{\gamma,1}\pare{x}=\frac{\sigma_{\gamma,1}\pare{x}+1}{2}+\frac{\sigma_{\gamma,1}\pare{x}-1}{2}$. Reading now the coefficient $1$ in front of the hyperbolic tangent as $\frac{1+\sigma_{\gamma,2}\pare{x}}{2}+\frac{1-\sigma_{\gamma,2}\pare{x}}{2}$, we get
\begin{align}
\begin{aligned}\label{36}
&L_\gamma\pic{\sigma_{\gamma,1}, G}
\\[5pt]
&\quad=-\pic{\sigma_{\gamma,1}, G}
\\[5pt]
&\quad\blanco{=}+\frac{\gamma}{2}\sum_{x\in\Lambda_\gamma}\llav{\tanh\corch{ \beta_1  \pare{\sigma_{\gamma,1}*\phi_1}\pare{x}+\beta_1\lambda}+\tanh\corch{ \beta_1  \pare{\sigma_{\gamma,1}*\phi_1}\pare{x}-\beta_1\lambda}}G(\gamma x)
\\[5pt]
&\quad\blanco{=}+\frac{\gamma}{2}\sum_{x\in\Lambda_\gamma}\sigma_{\gamma,2}\pare{x}\llav{\tanh\corch{ \beta_1  \pare{\sigma_{\gamma,1}*\phi_1}\pare{x}+\beta_1\lambda}-\tanh\corch{ \beta_1  \pare{\sigma_{\gamma,1}*\phi_1}\pare{x}-\beta_1\lambda}}G(\gamma x).
\end{aligned}
\end{align}
Analogously, we get identities
\begin{align}
\begin{aligned}\label{37}
&L_\gamma\pic{\sigma_{\gamma,2}, G}
\\[5pt]
&\quad=-\pic{\sigma_{\gamma,2}, G}
\\[5pt]
&\quad\blanco{=}+\frac{\gamma}{2}\sum_{x\in\Lambda_\gamma}\llav{\tanh\corch{ \beta_2 \pare{\sigma_{\gamma,2}*\phi_2}\pare{x}+\beta_2\lambda}+\tanh\corch{ \beta_2  \pare{\sigma_{\gamma,2}*\phi_2}\pare{x}-\beta_2\lambda}}G(\gamma x)
\\[5pt]
&\quad\blanco{=}-\frac{\gamma}{2}\sum_{x\in\Lambda_\gamma}\sigma_{\gamma,1}\pare{x}\llav{\tanh\corch{ \beta_2  \pare{\sigma_{\gamma,2}*\phi_1}\pare{x}+\beta_2\lambda}-\tanh\corch{ \beta_2 \pare{\sigma_{\gamma,2}*\phi_2}\pare{x}-\beta_2\lambda}}G(\gamma x)
\end{aligned}
\end{align}
and
\begin{equation}
\begin{split}\label{38}
&L_\gamma\pic{\sigma_{\gamma,1}\sigma_{\gamma,2},G} 
\\[5pt]
&\quad=-2\gamma\sum_{x\in \Lambda_\gamma}\sigma_{\gamma, 1}(x)\sigma_{\gamma, 2}(x)G(\gamma x) 
\\[5pt]
&\quad\blanco{=}+\frac{\gamma}{2}\sum_{x\in \Lambda_\gamma}\llav{\tanh\corch{\beta_1 \pare{\sigma_{\gamma,1}*\phi_1}\pare{x}+\beta_1\lambda}-\tanh\corch{\beta_1 \pare{\sigma_{\gamma,1}*\phi_1}\pare{x}-\beta_1\lambda}}G(\gamma x)
\\[5pt]
&\quad\blanco{=}-\frac{\gamma}{2}\sum_{x\in \Lambda_\gamma}\llav{\tanh\corch{\beta_2\pare{\sigma_{\gamma,2}*\phi_2}\pare{x}+\beta_2\lambda}-\tanh\corch{\beta_2 \pare{\sigma_{\gamma,2}*\phi_2}\pare{x}-\beta_2\lambda}}G(\gamma x)
\\[5pt]
&\quad\blanco{=}+\frac{\gamma}{2}\sum_{x\in \Lambda_\gamma}\sigma_{\gamma, 1}(x)\llav{\tanh\corch{\beta_2 \pare{\sigma_{\gamma,2}*\phi_2}\pare{x}+\beta_2\lambda}+\tanh\corch{\beta_2 \pare{\sigma_{\gamma,2}*\phi_2}\pare{x}-\beta_2\lambda}}G(\gamma x)
\\[5pt]
&\quad\blanco{=}+\frac{\gamma}{2}\sum_{x\in \Lambda_\gamma}\sigma_{\gamma, 2}(x)\llav{\tanh\corch{\beta_1 \pare{\sigma_{\gamma,1}*\phi_1}\pare{x}+\beta_1\lambda}+\tanh\corch{\beta_1\pare{\sigma_{\gamma,1}*\phi_1}\pare{x}-\beta_1\lambda}}G(\gamma x).
\end{split}
\end{equation}
The hydrodynamic limit is in evidence in the previous identities.

\subsubsection*{Control of the martingale terms}

We need some observations about the mathematical structure behind our Markov process.
If we consider the distance $\textbf{1}\llav{\ul\sigma\neq \tilde{\ul\sigma}}$ in the state space $E_\gamma$, it makes sense to define the set of \textit{c\`adl\`ag} functions $\CAL D\pare{\left[0,\infty\right),E_\gamma}$.
We think the Markov process $\ul\sigma_\gamma\pare{\cdot}$ as a $\CAL D\pare{\left[0,\infty\right),E_\gamma}$-random element (always defined in the same measure space $\pare{\Omega,\CAL F}$),
where $\CAL D\pare{\left[0,\infty\right),E_\gamma}$ is endowed with the $\sigma$-algebra generated by the projections $\ul\sigma_\gamma\pare{\cdot}\mapsto \ul\sigma_\gamma\pare{t}$.
For fixed $\gamma$, the stochastic process $\pare{\ul\sigma_\gamma\pare{t}}_{t\ge 0}$
generates the filtration $\pare{\CAL F_{\gamma}\pare{t}}_{t\ge 0}$ in $\pare{\Omega,\CAL F}$ defined by $\CAL F_\gamma\pare{t}\defi\sigma\llav{\ul\sigma_\gamma\pare{s}:0\le s\le t}$.

For $G\in C\pare{\BB T,\BB R}$, the processes
\begin{align}
\begin{aligned}\nonumber
&M^{G}_{\gamma,i}\pare{t}\defi \pic{\sigma_{\gamma,i}\pare{t},G}-\pic{\sigma_{\gamma,i}\pare{0},G}
-\int_0^tL_\gamma\pic{\sigma_{\gamma,i}(s),G}ds, \quad i\in\llav{1,2}
\\[5pt]
& M^{G}_{\gamma,3}\pare{t}\defi \pic{\eta_{\gamma}\pare{t},G}-\pic{\eta_{\gamma}\pare{0},G}
-\int_0^tL_\gamma\pic{\eta_{\gamma}(s),G}ds
\end{aligned}	
\end{align}
are martingales associated to the filtration $\pare{\CAL F_{\gamma}\pare{t}}_{t\ge 0}$ (see lemma 5.1 in the appendix 1 of \cite{KL99}, for instance).

\begin{proposition}\label{ddkk}
There exists a constant $C$ depending only on the macroscopic parameters $\llav{\beta_1,\beta_2,\lambda}$ and on the functions $\phi_1, \phi_2$ such that
\begin{align}\nonumber
\BB P\pare{\sup_{t\in\corch{0,T}}\abs{M_{\gamma,i}^{G}\pare{t}}>\zeta}\le
\frac{C\norm{G}_\infty T\gamma}{\zeta^2}
\end{align}
for every  $\zeta>0$ and  $G\in C\pare{\BB T,\BB R}$.
\end{proposition}

\begin{proof}[Proof of proposition \ref{ddkk}]
We prove the proposition only for $i=1$ as the other cases are similar.
By Doob's inequality for martingales, 
\begin{align}\nonumber
\BB P\pare{\sup_{t\in\corch{0,T}}\abs{M_{\gamma,1}^{G}\pare{t}}>\zeta}\le
\frac{\BB E\corch{\pare{M_{\gamma,1}^{G}\pare{T}}^2}}{\zeta^2}
\end{align}
for every $\zeta>0$. By lemma 5 in the appendix 1 of \cite{KL99}, we know that the process 
\begin{align}\label{4567}
N_{\gamma,1}^{G}(t):=\pare{M_{\gamma,1}^{G}(t)}^2-\int_0^tL_\gamma\pic{\sigma_{\gamma,1}(s),G}^2-2\pic{\sigma_{\gamma,1}(s),G}L_\gamma\pic{\sigma_{\gamma,1}(s),G}ds
\end{align}
is a zero-mean martingale associated to the filtration $\pare{\CAL F_{\gamma}\pare{t}}_{t\ge 0}$.
Therefore, to prove the proposition, it is enough to estimate the expectation of the integral appearing in \eqref{4567}.
Easy computations show that, for every $\ul\sigma_\gamma\in E_\gamma$,
\begin{align}\nonumber
\begin{aligned}\nonumber
L_\gamma\pic{\sigma_{\gamma,1},G}^2-2\pic{\sigma_{\gamma,1},G}L_\gamma\pic{\sigma_{\gamma,1}(s),G}&=\sum_{x\in\Lambda_\gamma}R_1\pare{x, \ul\sigma_{\gamma}}\pare{\pic{\sigma_{\gamma,1}^x,G}-\pic{\sigma_{\gamma,1},G}}^2
\\[5pt]
&\leq \sum_{x\in\Lambda_\gamma}R_1\pare{x, \ul\sigma_{\gamma}}4\|G\|_\infty\gamma^2\leq C\|G\|_\infty\gamma,
\end{aligned}
\end{align}
inequality that let us conclude.
\end{proof}

\subsubsection*{Tightness and uniqueness of the limit point}

For every $\gamma$, let $P^\gamma$ be the law of the empirical measures associated to our Markov process in the time interval $\corch{0,T}$ (it will be properly defined in the next paragraph).
We will prove the convergence in distribution of the sequence $\llav{P^\gamma}_\gamma$ to the Dirac measure concentrated on a deterministic continuous path $\ul\pi^*\pare{\cdot}=\pare{\pi^*_{1}\pare{\cdot},\pi^*_{2}\pare{\cdot},\pi^*_3\pare{\cdot}}$, where $\pi_1^*(t)$, $\pi_2^*(t)$ and $\pi_3^*(t)$ are the signed-measures on $\BB T$ absolutely continuous with respect to the Lebesgue measure with densities $u_1(t,r)$, $u_2(t,r)$, $v(t,r)$ satisfying \eqref{hydro1}, \eqref{hydro2} and \eqref{hydro3} respectively.
Later we will argue that the convergence in distribution to a deterministic continuous trajectory implies uniform convergence in probability in the sense stated in  theorem \eqref{hydrolimit}.
In order to prove the convergence of the sequence $\llav{P^\gamma}_\gamma$, we proceed in two main steps: we prove that this sequence is tight and that all converging subsequences converge to the same limit.

We need to introduce some background before defining $P^\gamma$.
All the facts used in the rest of this paragraph are well known and can be found, for instance, in \cite{Bil99,Bog07,BB04}.
Let $\CAL M\defi \llav{\pi\tn{ finite signed-measure on }\BB T}$ and
$\CAL M_1\defi \llav{\pi\in\CAL M:\abs{\pi}\pare{\BB T}\le 1}$.
We endow $\CAL M_1$ with the weak topology $\tau_{\CAL M_1}^*$, that is the one generated by the family of fundamental neighborhoods
\begin{align}\nonumber
\begin{aligned}\nonumber
&\left\{\llav{\tilde\pi\in\CAL M_1:\abs{\pic{\tilde\pi,G_i}-\pic{\pi,G_i}}\le\varepsilon \ \forall i=1,\ldots,n}:\right. \\
&\quad\quad\quad\quad\quad\quad \quad \quad \quad \left.\pi\in\CAL M_1,\varepsilon>0,n\in\BB N,G_1,\ldots,G_n\in C\pare{\BB T,\BB R}\right\}.
\end{aligned}
\end{align}
In this topology, a net $\pare{\pi_\alpha}_\alpha$ converges to $\pi$ if and only if the net $\pare{\pic{\pi_\alpha,G}}_\alpha$ converges to $\pic{\pi,G}$ for every $G\in C\pare{\BB T,\BB R}$, denoting $\pic{\pi,G}$ the integral $\int Gd\pi$.
Fix from now on a $\norm{\cdot}_\infty$-dense sequence $\pare{f_n}_{n\in\BB N}$ in $C\pare{\BB T,\BB R}$.
$\tau_{\CAL M_1}^*$ can be metrized by the metric defined as
\begin{align}\nonumber
d\pare{\pi,\tilde\pi}\defi \sum_{n\in\BB N}2^{-n}\frac{\abs{\pic{\pi,f_n}-\pic{\tilde\pi,f_n}}}{1+\abs{\pic{\pi,f_n}-\pic{\tilde\pi,f_n}}}.
\end{align}
The metric space $\pare{\CAL M_1,d}$ is a Polish space.
Let $\CAL D=\CAL D\pare{\corch{0,T},\CAL M_1}$ be the \textit{c\`adl\`ag} space associated to $\CAL M_1$.
Elements of $\CAL D$ will be denoted as $\pi\pare{\cdot}$.
We endow $\CAL D$ with the modified Skorohod metric
\begin{align}\nonumber
\rho\pare{\pi\pare{\cdot},\tilde\pi\pare{\cdot}}\defi \inf_{\lambda\in\Lambda}\llav{\norm{\lambda}^{\circ}\vee \sup_{t\in\corch{0,T}}d\pare{\pi\pare{t},\tilde\pi\pare{\lambda t}}},
\end{align}
where $\Lambda$ is the set of non-decreasing functions from $\corch{0,T}$ to itself such that $\lambda 0=0$ and $\lambda T=T$, and 
\begin{align}\nonumber
\norm{\lambda}^{\circ}\defi \sup_{0\le s< t\le T}\abs{\log \frac{\lambda t-\lambda s}{t-s}}.
\end{align}
The metric space $\pare{\CAL D,\rho}$ is again a Polish space.
We endow it with the Borel $\sigma$-algebra (the one generated by $\rho$), and we endow the product space $\CAL D^3$ with the product of these Borel $\sigma$-algebras.
Elements of $\CAL D^3$ will be denoted as $\ul\pi\pare{\cdot}=\pare{\pi_1\pare{\cdot},\pi_2\pare{\cdot},\pi_3\pare{\cdot}}$.
There is a natural embedding $\chi:\llav{-1,1}^{\Lambda_\gamma}  \hookrightarrow\CAL M_1$: every spin configuration is mapped into its associated empirical signed-measure.
$\chi$ induces an embedding $\Gamma:\CAL D\pare{\left[0,\infty\right),E_\gamma}\hookrightarrow \CAL D^3$:
\begin{align}\nonumber
\ul\sigma_\gamma\pare{\cdot}\mapsto \pare{\pare{\chi\sigma_{\gamma,1}\pare{t},\chi\sigma_{\gamma,2}\pare{t},\chi\eta_{\gamma}\pare{t}}}_{t\in\corch{0,T}}.
\end{align}
The composition $\Gamma\circ\ul\sigma_\gamma\pare{\cdot}:\Omega\to \CAL D^3$ is measurable, so it makes sense to define the probability $P^\gamma$ on $\CAL D^3$ induced by it.

For $i\in\llav{1,2,3}$, let $P^\gamma_{i}$ be the $i$-th marginal of $P^\gamma$.
The tightness of $\pare{P^\gamma}_\gamma$ follows from the tightness of $\pare{P^\gamma_{i}}_\gamma$ for every $i$.
We only prove the tightness of $\pare{P^\gamma_{1}}_\gamma$ as the cases  $i\in\llav{2,3}$ are similar.

As stated in chapter 4 of \cite{KL99}, the tightness of $\pare{P^\gamma_{1}}_\gamma$ follows from the following conditions:
\begin{enumerate}
\item for every $t\in [0,T]$ and $ \varepsilon>0$, there exists a compact set $K=K(t,\varepsilon )\subset \CAL M_1$ such that
\begin{align}\nonumber
\sup_{\gamma} P^\gamma_{1}\pare{\pi\pare{\cdot}:\pi(t)\notin K\pare{t,\varepsilon}}\leq \varepsilon;
\end{align}
\item for every $\varepsilon>0$,
\begin{align}\nonumber
\lim_{\delta\to 0}
\limsup_{\gamma\to 0}
P^\gamma_{1}\pare{\pi\pare{\cdot}:
\sup_{\abs{t-s}\le \delta}d\pare{\pi\pare{t},\pi\pare{s}}
>\varepsilon}=0.
\end{align}
\end{enumerate}
The first condition is automatically satisfied as $\CAL M_1$ is compact (use theorem 3.3 of \cite{Var58} and the fact that $d$ metrizes the weak topology in $\CAL M_1$).
The second condition requires more effort.
We will prove the following stronger assertion: if $\delta$ is small enough,
\begin{align}\nonumber
\lim_{\gamma\to 0}P^\gamma_{1}\pare{\pi\pare{\cdot}:\sup_{\abs{t-s}\le\delta}d\pare{\pi\pare{t},\pi\pare{s}}\le \varepsilon}=
\lim_{\gamma\to 0}\BB P\pare{\sup_{\abs{t-s}\le\delta}d\pare{\sigma_{\gamma,1}\pare{t},\sigma_{\gamma,1}\pare{s}}\le \varepsilon}=1.
\end{align}
By an abuse of notation, we are using the metric $d$ in the space $\llav{-1,1}^{\Lambda_\gamma}$: $d\pare{\sigma_\gamma,\tilde\sigma_\gamma}=d\pare{\chi\sigma_\gamma,\chi\tilde\sigma_\gamma}$. 
Take $N\in\BB N$ such that $2^{-N}<\frac{\varepsilon}{3}$.
Then
\begin{align}\nonumber
d\pare{\sigma_{\gamma,1}\pare{t},\sigma_{\gamma,1}\pare{s}}\le \sum_{n=1}^N\abs{\pic{\sigma_{\gamma,1}\pare{t},f_n}-\pic{\sigma_{\gamma,1}\pare{s},f_n}}+\frac{\varepsilon}{3}.
\end{align}
An easy calculation shows that
there exists a constant $C_1$ depending only on the macroscopic parameters $\llav{\beta_1,\beta_2,\lambda}$ and on the functions $\phi_1, \phi_2$ such that inequality
\begin{align}
\nonumber
L_\gamma\pic{\sigma_{\gamma},G}\le C_1\norm{G}_{\infty}
\end{align}
holds for every $G\in C\pare{\BB T,\BB R}$ and every $\sigma_{\gamma}\in\llav{-1,1}^{\Lambda_\gamma}$;
then, for every $n$,
\begin{align}
\begin{aligned}\nonumber
\abs{\pic{\sigma_{\gamma,1}\pare{t},f_n}-\pic{\sigma_{\gamma,1}\pare{s},f_n}}&\le \abs{\int_s^tL_\gamma\pic{\sigma_{\gamma,1}\pare{\tilde s},f_n}d\tilde s}+\abs{M_\gamma^{f_n,1}\pare{t}-M_\gamma^{f_n,1}\pare{s}} \\[5pt]
&\le \abs{t-s}C_1\norm{f_n}_\infty+\abs{M_\gamma^{f_n,1}\pare{t}-M_\gamma^{f_n,1}\pare{s}}.
\end{aligned}
\end{align}
Chose $\delta$ such that $\dis{\delta NC_1\max_{n=1}^N\norm{f_n}_\infty<\frac{\varepsilon}{3}}$.
Under this choice, if we are in the event $\dis{\bigcap_{n=1}^N\pare{\sup_{\abs{t-s}\le \delta}\abs{M_\gamma^{f_n,1}\pare{t}-M_\gamma^{f_n,1}\pare{s}}\le \frac{\varepsilon}{3N}}}$, we have
$\dis{\sup_{\abs{t-s}\le \delta}d\pare{\sigma_{\gamma,1}\pare{t},\sigma_{\gamma,1}\pare{s}}<\varepsilon}$.
Then we only need to prove that the probability of the previous event goes to $1$ as $\gamma\to 0$;
this follows if we prove that, for fixed $n\in\llav{1,\ldots,N}$, $\dis{\BB P\pare{\sup_{\abs{t-s}\le \delta}\abs{M_\gamma^{f_n,1}\pare{t}-M_\gamma^{f_n,1}\pare{s}}\le \frac{\varepsilon}{3N}}\conv{\gamma\to 0}1}$;
this follows by taking $\varepsilon=\gamma^{\frac{1}{4}}$ and $G=f_n$ in proposition \ref{ddkk}.

Once we know that the sequence $P^\gamma$ is tight, it remains to characterize all its limit points.
Let $P^*$ be a limit point and $P^{\gamma_k}$ be a subsequence converging to $P^*$.
The following two lemmas show that $P^*$ is concentrated on continuous trajectories $ \ul\pi\pare{\cdot}$ that are absolutely continuous with respect to the Lebesgue measure.

\begin{lemma}\label{rrrrr}
$P^*\pare{C\pare{[0,T], \mathcal{M}_1}^3}=1.$
\end{lemma}
\begin{proof}
For $i\in\llav{1,2,3}$, let $P^{\gamma_k}_i$ and $P^*_{i}$ respectively be the $i$-th marginal of $P^{\gamma_k}$ and $P^*$.
Observe that $P^{\gamma_k}_{i}$ converges to $P^*_{i}$.
In order to prove the lemma, it is enough to show that $P^*_{i}$ is concentrated on continuous trajectories for all $i\in\llav{1,2,3}$.
We will show it just for $i=1$; the proofs for the other cases are similar.
Let $\Delta: \CAL D\to \BB R$ be the map defined as
\begin{displaymath}
\Delta\pare{\pi\pare{\cdot}}=\sup_{t\in [0,T]}d\pare{\pi(t),\pi(t-)}.
\end{displaymath}
$\Delta$ is the supremum of the jumps of the process $\pi\pare{\cdot}$.
The lemma follows once we have proven that
$\BB E_{P^*_{1}}\pare{\Delta}=0$.
As the map $\Delta$ is continuous, it is enough to show that 
\begin{align}\nonumber
\lim_{\gamma_k\to 0}\mathbb{E}_{P^{\gamma_k}_{1}}\pare{\Delta}=0.
\end{align}
Observe that, for $\ul\sigma_{\gamma_k}\in E_{\gamma_k}$, $t\in [0,T]$ and $N\in\BB N$,
\begin{align}\nonumber
d\pare{\sigma_{\gamma_k,1}(t),\sigma_{ 1}(t-)}
\leq \sum_{n=1}^{N}\abs{\pic{\sigma_{\gamma_k,1}(t), f_n}-\pic{\sigma_{\gamma_k,1}(t-), f_n}}+2^{-N}\leq 2\sum_{n=1}^{N}\gamma_k\|f_n\|_{\infty}+2^{-N}.
\end{align}
This implies that
\begin{align}\nonumber
\begin{aligned}
\mathbb{E}_{P^1_{\gamma_k}}\pare{\Delta\pare{\Pi}}\leq 2\sum_{n=1}^{N}\gamma_k\|f_n\|_{\infty}+2^{-N}.
\end{aligned}
\end{align}
Let $\gamma_k\to 0$  and $N\to\infty$ to conclude.
\end{proof}
\begin{lemma}
$P^*$ is concentrated on trajectories $\ul\pi\pare{\cdot}$ such that $\pi_i(t)$ is absolutely continuous with respect to the Lebesgue measure for every $t\in [0,T]$ and $i\in\llav{1,2,  3}$.
\end{lemma}
\begin{proof}
A signed-measure $\pi$  is absolutely continuous with respect to the Lebesgue measure if 
\begin{align}\nonumber
\abs{\pic{\pi, G}}\leq\int_\BB T G(r)dr\qquad \forall G\in C\pare{\BB T, \BB R^+}.
\end{align}
Then the assertion follows if we prove that, for all $i\in\llav{1,2, 3}$,
\begin{align}\nonumber
P^*\pare{\sup_{t\in [0,T]}\abs{\pic{\pi_i(t), G}}\leq\int_\BB T G(r)dr \quad\forall G\in C\pare{\BB T, \BB R^+}}=1.
\end{align}
We will do it just for $i=1$ as the other cases are analogous.
As $C\pare{\BB T, \BB R^+}$ is $\norm{\cdot}_\infty$-separable, it is enough to show that,
for all $G\in C\pare{\BB T, \BB R^+}$,
\begin{align}\nonumber
P^*\pare{\sup_{t\in [0,T]}\abs{\pic{\pi_1(t), G}}\leq\int_\BB T G(r)dr}=1.
\end{align}
Since
\begin{align}\nonumber
P^{\gamma_k}\pare{\sup_{t\in [0,T]}\abs{\pic{\pi_1(t), G}}\leq\gamma_k\sum_{x\in\Lambda_{\gamma_k}}G(\gamma_k x)}=1
\end{align}
and the map $\dis{\ul\pi\pare{\cdot}\mapsto\sup_{t\in\corch{0,T}} \abs{\pic{\pi_1(t), G}}}$ is continuous,  we get
\begin{align}
\begin{aligned}
\nonumber
&P^*\pare{\sup_{t\in [0,T]}\abs{\pic{\pi_1(t), G}}-\int_{\BB T}G(r)dr>\varepsilon}
\\[5pt]
& \quad \leq\liminf_{\gamma_k\to 0}P^{\gamma_k}\pare{\sup_{t\in [0,T]}\abs{\pic{\pi_1(t), G}}-\int_{\BB T}G(r)dr>\varepsilon}
\\[5pt]
& \quad \leq\liminf_{\gamma_k\to 0}P^{\gamma_k}\pare{\sup_{t\in [0,T]}\abs{\pic{\pi_1(t), G}}-\gamma_k\sum_{x\in\Lambda_{\gamma_k}}G(\gamma_k x)>\frac{\varepsilon}{2}}=0
\end{aligned}
\end{align}
for every $\varepsilon>0$.
\end{proof}

Before continuing, we introduce some definition to have compact notations.
For all $\sigma_{\gamma}\in\llav{-1,1}^{\Lambda_\gamma}$ and $i\in\llav{1,2}$, we define the functions $W^{\pm,i}_{\sigma_{\gamma}}:\BB T\to\BB R$ as
\begin{displaymath}
W^{\pm,i}_{\sigma_{\gamma}}\pare{r}\defi
\frac{1}{2}\tanh\corch{\beta_i\pare{\pic{\sigma_{\gamma},\phi_i\pare{r,\cdot}}\pm \lambda}}.
\end{displaymath}
Let $\mathscr{C}_\gamma$ be the counting measure defined as $\mathscr{C}_\gamma\pare{dr}=\gamma\sum_{x\in\lambda_\gamma}\delta_{\gamma x}\pare{dr}$, for $G\in C\pare{\BB T, \BB R}$ and $i\in\llav{1,2,3}$,
we define the functions $B_{\gamma,i}^{G}:E_\gamma \to\BB R$ as
\begin{align}
\begin{aligned}\nonumber
B_{\gamma,1}^{G}\pare{\ul\sigma_\gamma}\defi&-\pic{\sigma_{\gamma,1}, G}+\pic{\sigma_{\gamma,2}, \pare{W^{+,1}_{\sigma_{\gamma,1}}- W^{-,1}_{\sigma_{\gamma,1}}}G}+\pic{\mathscr{C}_\gamma, \pare{W^{+,1}_{\sigma_{\gamma,1}}+ W^{-,1}_{\sigma_{\gamma,1}}}G}
\\[5pt]
B_{\gamma,2}^{G}\pare{\ul\sigma_{\gamma}}\defi&-\pic{\sigma_{\gamma,2}, G}
-\pic{\sigma_{\gamma,1}, \pare{W^{+,2}_{\sigma_{\gamma,2}}
- W^{-,2}_{\sigma_{\gamma,2}}}G}+\pic{\mathscr{C}_\gamma, \pare{W^{+,2}_{\sigma_{\gamma,2}}+ W^{-,2}_{\sigma_{\gamma,2}}}G}
\\[5pt]
B_{\gamma,3}^{G}\pare{\ul\sigma_{\gamma}}\defi &-2\pic{\sigma_{\gamma,1}\sigma_{\gamma,2}, G}
+\pic{\sigma_{\gamma,1}, \pare{W^{+,2}_{\sigma_{\gamma,2}}
+W^{-,2}_{\sigma_{\gamma,2}}}G}+\pic{\sigma_{\gamma,2}, \pare{W^{+,1}_{\sigma_{\gamma,1}}
+W^{-,1}_{\sigma_{\gamma,1}}}G}
\\[5pt]
&+\pic{\mathscr{C}_\gamma, \pare{W^{+,1}_{\sigma_{\gamma,1}}-W^{-,1}_{\sigma_{\gamma,1}}-W^{+,2}_{\sigma_{\gamma,2}}+W^{-,2}_{\sigma_{\gamma,2}}}G}.
\end{aligned}
\end{align}
The previous expressions respectively represent the right-hand sides of identities
(\ref{36}-\ref{38}).
Then 
we can write
\begin{align}\label{ser}
\begin{aligned}
M^{G}_{\gamma,i}\pare{t}=& \pic{\sigma_{\gamma,i}\pare{t},G}-\pic{\sigma_{\gamma,i}\pare{0},G}-\int_0^t B_{\gamma,i}^{G}\pare{\ul\sigma_{\gamma}\pare{s}}ds
\end{aligned}
\end{align}
for $i\in\llav{1,2}$, and
	\begin{align}\nonumber
\begin{aligned}
M^{G}_{\gamma,3}\pare{t}=& \pic{\eta_{\gamma}\pare{t},G}-\pic{\eta_{\gamma}\pare{0},G}-\int_0^t B_{\gamma,3}^{G}\pare{\ul\sigma_{\gamma}\pare{s}}ds.
\end{aligned}
\end{align}
The definition of the function $W^{\pm,i}_{\sigma_\gamma}$ makes sense if, in the subindex, we put a signed-measure $\pi$ instead of a spin configuration $\sigma_{\gamma}$:
$W^{\pm,i}_{\pi}\pare{r}\defi \frac{1}{2}\tanh\corch{\beta_i\pare{\pic{\pi,\phi_i\pare{r,\cdot}}\pm \lambda}}$.
In the same way, $B_{\gamma,i}^{G}$ can be read as a function from $\CAL M_1^3$ to $\BB R$.
For $G\in C\pare{\BB T, \BB R}$, $\ul\pi\in\CAL M_1^3$ and $i\in\llav{1,2,3}$, we define $B_i^{G}:\CAL M_1^3\to\BB R$ as
\begin{align}\label{uiop}
\begin{aligned}
B_1^{G}\pare{\ul\pi}\defi&-\pic{\pi_{1},G}+\pic{\pi_2, \pare{W^{+,1}_{\pi_1}- W^{-,1}_{\pi_1}}G}+\pic{\mathscr{L}, \pare{W^{+,1}_{\pi_1}+ W^{-,1}_{\pi_1}}G}\\[5pt]
B_2^{G}\pare{\ul\pi}\defi&-\pic{\pi_{2},G}-\pic{\pi_1, \pare{ W^{+,2}_{\pi_2}- W^{-,2}_{\pi_2}}G}+\pic{\mathscr{L}, \pare{ W^{+,2}_{\pi_2}+ W^{-,2}_{\pi_2}}G}\\[5pt]
B_3^{G}\pare{\ul\pi}\defi&-2\pic{\pi_3, G}+\pic{\pi_1, \pare{W^{+,2}_{\pi_2}+W^{-,2}_{\pi_2}}G}+\pic{\pi_2, \pare{W^{+,1}_{\pi_1}+W^{-,1}_{\pi_1}}G}\\[5pt]
&+\pic{\mathscr{L}, \pare{W^{+,1}_{\pi_1}-W^{-,1}_{\pi_1}-W^{+,2}_{\pi_2}+W^{-,2}_{\pi_2}}G},
\end{aligned}
\end{align}
$\mathscr{L}$ denoting the Lebesgue measure in $\BB T$.
$B_i^{G}$ has to be read  as the limit as $\gamma\to 0$ of $B_{\gamma,i}^{G}$  ($\mathscr C_\gamma$ converges to the Lebesgue measure in the weak sense).

\begin{lemma}\label{nnb}
For $G\in C\pare{\BB T, \BB R}$ and $i\in\llav{1,2,3}$,
the map $\Phi_i^G:\CAL D^3\to\BB R$ defined as
\begin{align}\nonumber
\Phi_i^G(\ul\pi\pare{\cdot})\defi\sup_{t\in [0,T]}\abs{\pic{\pi_{i}\pare{t},G}-\pic{\pi_{i}\pare{0},G}-\int_0^tB_i^{G}\pare{\ul\pi(s)}ds}
\end{align}
is continuous in the elements of $C\pare{[0,T], \CAL M_1}^3$.
\end{lemma}
\begin{proof}
Consider a sequence $\ul\pi^{n}\pare{\cdot}\in \CAL D^3$ such that
$\ul\pi^{n}\pare{\cdot}\xrightarrow[n\to\infty]{}\ul \pi\pare{\cdot}\in C\pare{[0,T], \CAL M_1}^3$.
As we are taking the product topology, the previous convergence means convergence in each coordinate.
Since $\rho\pare{\pi_i^{n}\pare{\cdot}, \pi_i\pare{\cdot}}\xrightarrow[n\to\infty]{}0$ and $\pi_i\pare{\cdot}\in C\pare{[0,T], \CAL M_1}$ for $i\in\llav{1,2,3}$, it follows that $\dis{\sup_{t\in [0,T]}d\pare{\pi^{n}_{i}(t), \pi_{i}(t)}\xrightarrow[n\to \infty]{}0}$ (convergence in the modified Skorohod metric to a continuous function implies uniform convergence) and, consequently,
\begin{align}\label{hu}
\sup_{t\in [0,T]}\abs{\pic{\pi^{n}_{i}(t), G}-\pic{\pi_i(t), G}}\xrightarrow[n\to\infty]{}0.
\end{align}
To see why the last assertion is true observe that, for all $\varepsilon>0$, there exists a function $f_{\bar k}$ such that $\norm{G-f_{\bar k}}_\infty<\varepsilon$ and, consequently,
\begin{align}\label{dsa}
\begin{aligned}
\sup_{t\in [0,T]}\abs{\pic{\pi^{n}_{i}(t), G}-\pic{\pi_i(t), G}}&\leq \sup_{t\in [0,T]}\abs{\pic{\pi^{n}_{i}(t), G-f_{\bar k}}-\pic{\pi_i(t), G-f_{\bar k}}}\\
&\blanco{\leq}+\sup_{t\in [0,T]}\abs{\pic{\pi^{n}_{i}(t), f_{\bar k}}-\pic{\pi_i(t), f_{\bar{k}}}}\\
&\leq \varepsilon + 2^{\bar k}\pare{1+\norm{f_{\bar k}}_\infty}\sup_{t\in [0,T]}d\pare{\pi^{n}_{i}(t), \pi_{i}(t)}
\end{aligned}
\end{align}
for all $i\in\llav{1,2,3}$;
\eqref{hu} is proved after taking $n\to\infty$ and $\varepsilon\to 0$ in \eqref{dsa}.
The proof of the lemma follows once we show that $\Phi^G_i\pare{\ul\pi^{n}\pare{\cdot}}\xrightarrow[n\to\infty]{}\Phi^G_i\pare{\ul\pi\pare{\cdot}}$.
We will show it only for $i=1$ as the other cases are similiar.
Observe that
\begin{align}\nonumber
\begin{aligned}
&\abs{\Phi_1^G\pare{\ul\pi^{n}\pare{\cdot}}-\Phi_1^G\pare{\ul\pi\pare{\cdot}}}
\\[5pt]
&\quad\leq  2\sup_{t\in [0,T]}\abs{\pic{\pi^{n}_{1}(t), G}-\pic{\pi_{1}(t), G}}
+\int_0^T\abs{B^{G}_i\pare{\ul\pi^{n}(s)}-B^{G}_1\pare{\ul\pi(s)}}ds.
\end{aligned}
\end{align}
Since the first term in the right-hand side converges to $0$ because of \eqref{hu}, the proof is concluded by showing that
\begin{align}\nonumber
\int_0^T\abs{B^{G}_1\pare{\ul\pi^{n}(s)}-B^{G}_1\pare{\ul\pi(s)}}ds\xrightarrow[n\to\infty]{}0.
\end{align}
By the dominated convergence theorem,  it is enough to show that, for every $s\in[0,T]$, 
\begin{align}\nonumber
\abs{B^{G}_1\pare{\ul\pi^{n}(s)}-B^{G}_1\pare{\ul\pi(s)}}\xrightarrow[n\to+\infty]{}0.
\end{align}
As $s$ is fixed, we omit writing it in the following computations:
\begin{align}\nonumber
\begin{aligned}
&\abs{B^{G}_1\pare{\ul\pi^{n}}-B^{G}_1\pare{\ul\pi}}
\\[5pt]
&\quad\leq  \abs{\pic{\pi^{n}_{1},G}-\pic{\pi_1, G}}\\[5pt]
&\quad\blanco{\le}+\abs{\pic{\pi^{n}_{2}, W^{+,1}_{\pi^{n}_{1}}G}-\pic{\pi_2, W^{+,1}_{\pi_1}G}}+\abs{\pic{\pi^{n}_{2}, W^{-,1}_{\pi^{n}_{1}}G}-\pic{\pi_2, W^{-,1}_{\pi_1}G}}\\[5pt]
&\quad\blanco{\le}+\abs{\pic{ \mathscr L, W^{+,1}_{\pi^{n}_{1}}G}-\pic{\mathscr L, W^{+,1}_{\pi_1}G}}+\abs{\pic{\mathscr L, W^{-,1}_{\pi^{n}_{1}}G}-\pic{\mathscr L, W^{-,1}_{\pi_1}G}}\\[5pt]
&\quad\leq\abs{\pic{\pi^{n}_{1}, G}-\pic{\pi_1, G}}+\|G\|_\infty\abs{\pic{\pi_2^n, W_{\pi_1}^{+,1}-W_{\pi_1}^{-,1}}-\pic{\pi_2, W_{\pi_1}^{+,1}-W_{\pi_1}^{-,1}}}\\[5pt]
&\quad\blanco{\le}+2\beta_1\|G\|_\infty\sup_{r\in\BB T}\abs{\pic{\pi^{n}_{1}, \phi_{1}(r, \cdot)}-\pic{\pi_{1}, \phi_{1}(r, \cdot)}}.
\end{aligned}
\end{align}
The first and the second term in the right-hand side converge to $0$ as $n$ goes to $\infty$.
To prove that also the third one vanishes, we divide the macroscopic torus in $h^{-1}$ intervals of length $h$  and call $I_j$ the $j$-th interval centered in $r_j$.
Fixing $r\in\BB T$ and letting $I_{k}$ be the interval in which $r$ is contained, we have
\begin{align}\nonumber
\begin{aligned}
\abs{\pic{\pi^{n}_{1}, \phi_{1}(r, \cdot)}-\pic{\pi_{1}, \phi_{1}(r, \cdot)}}\leq
&\abs{\pic{\pi^{n}_{1}, \phi_{1}(r, \cdot)}-\pic{\pi^{n}_{1}, \phi_{1}(r_{k}, \cdot)}}
\\[5pt]
&+\abs{\pic{\pi^{n}_{1}, \phi_{1}(r_{k}, \cdot)}-\pic{\pi_{1}, \phi_{1}(r_{k}, \cdot)}}
\\[5pt]
&+\abs{\pic{\pi_{1}, \phi_{1}(r_{k}, \cdot)}-\pic{\pi_{1}, \phi_{1}(r, \cdot)}}
\\[5pt]
\leq &2h+\abs{\pic{\pi^{n}_{1}, \phi_{1}(r_{k}, \cdot)}-\pic{\pi_{1}, \phi_{1}(r_{k}, \cdot)}}.
\end{aligned}
\end{align}
Then
\begin{align}\nonumber
\sup_{r\in\BB T}\abs{\pic{\pi^{n}_{1}, \phi_{1}(r, \cdot)}-\pic{\pi_{1}, \phi_{1}(r, \cdot)}}&\leq 2h+\max_{j\in\llav{1,\ldots, h^{-1}}}\abs{\pic{\pi^{n}_{1}, \phi_{1}(r_{ j}, \cdot)}-\pic{\pi_{1}, \phi_{1}(r_{ j}, \cdot)}},
\end{align}
which converges to $0$ letting first $n\to\infty$ and after $h\to 0$.
This concludes the proof.
\end{proof}

\begin{proposition}
Let $G\in C\pare{\BB T, \BB R}$.
$P^*$  is concentrated on trajectories  $\ul\pi\pare{\cdot}$  such that, for all $i\in\llav{1,2,3}$ and $t\in [0,T]$, 
\begin{align}\nonumber
\pic{\pi_{i}\pare{t},G}=\pic{\pi_{i}\pare{0},G}+\int_0^tB^{G}_i\pare{\ul\pi(s)}ds
\end{align}
(recall the definition of $B^{G}_i$ in \eqref{uiop}).
\end{proposition}
\begin{proof}
We will prove the theorem only for $i=1$ as the other cases are similiar.
By \eqref{ser} and proposition \eqref{ddkk}, we can deduce that
\begin{align}\nonumber
P^{\gamma_k}\pare{\sup_{t\in [0,T]}\abs{\pic{\pi_{1}\pare{t},G}-\pic{\pi_{1}\pare{0},G}-\int_0^t  B_{\gamma_k, 1}^{G}\pare{\ul\pi(s)}ds}>2\varepsilon}\xrightarrow[\gamma_k\to 0]{}0.
\end{align}
Since
\begin{align}\nonumber
\sup\llav{\abs{\pic{\mathscr{C}_\gamma, F}-\pic{\mathscr{L}, F}}:
F\in C^1\pare{\BB T,\BB R} \tn{ such that }\max\llav{\|F\|_\infty, \|F'\|_\infty}\leq 1}
\xrightarrow[\gamma\to 0]{}0,
\end{align}
it follows that
\begin{align}\nonumber
P^{\gamma_k}\pare{\sup_{t\in [0,T]}\abs{\pic{\pi_{1}\pare{t},G}-\pic{\pi_{1}\pare{0},G}-\int_0^t  B^{G}_1\pare{\ul\pi(s)}ds}>\varepsilon}\xrightarrow[\gamma_k\to 0]{}0.
\end{align}
By lemma \eqref{nnb}, the function $\Phi_1^G$ is continuous in the elements of $C\pare{[0,T], \CAL M}^3$; then, using the continuous map theorem (Theorem 2.7 of \cite{Bil68}, for instance) and Lemma \eqref{rrrrr},  we can conclude that
\begin{align}\nonumber
\begin{aligned}
&P^*\pare{\sup_{t\in [0,T]}\abs{\pic{\pi_{1}\pare{t},G}-\pic{\pi_{1}\pare{0},G}-\int_0^tB^{G}_1\pare{\ul\pi(s)}ds}>\varepsilon}\\[5pt]
&\quad\leq\liminf_{\gamma_k\to 0}P^{\gamma_k}\pare{\sup_{t\in [0,T]}\abs{\pic{\pi_{1}\pare{t},G}-\pic{\pi_{1}\pare{0},G}-\int_0^tB^{G}_1\pare{\ul\pi(s)}ds}>\varepsilon}=0,
\end{aligned}
\end{align}
which implies our claim.
\end{proof}

Thus every limit point of the sequence $P^\gamma$ is concentrated on continuous trajectories $\ul\pi\pare{\cdot}$ whose densities $u_1(t,r), u_2(t,r), v(t,r)$ respectively satisfy  equations \eqref{hydro1}-\eqref{hydro3}.
For every $\varepsilon>0$ and $i\in\llav{1,2}$, 
\begin{align}\nonumber
\begin{aligned}
&P^*\pare{\abs{\pic{\pi_i(0), G}-\int_{\BB T}G(r)\psi_i(r)dr}>\varepsilon} \\[5pt]
&\ \ \ \ \ \leq\liminf_{\gamma_k\to 0}P^{\gamma_k}\pare{\abs{\pic{\pi_i(0),G}-\int_{\mathbb{T}}G(r)\psi_i(r)dr}>\varepsilon}\\[5pt]
&\ \ \ \ \ =\lim_{\gamma_k\to 0}\BB P\pare{\abs{\gamma_k\sum_{x\in \Lambda_{\gamma_\kappa}}\sigma_{\gamma,i}(x)G(\gamma_k x)-\int_{\mathbb{T}^d}G(r)\psi_i(r)dr}>\varepsilon}=0
\end{aligned}
\end{align}
and, in the same way,
\begin{align}\nonumber
\begin{aligned}
&P^*\pare{\abs{\pic{\pi_3(0), G}-\int_{\BB T}G(r)\psi_1(r)\psi_2(r)dr}>\varepsilon}=0;
\end{aligned}
\end{align}
then we can conclude that $u_i(0,r)=\psi_i(r)$ for $i\in\llav{1,2}$ and $v(0,r)=\psi_1(r)\psi_2(r)$. The existence and uniqueness of the solution of system \eqref{hydro1}-\eqref{hydro3} guarantees the uniqueness of the limit point $P^*$, which is concentrated on the  trajectory $\ul \pi^*\pare{\cdot}$ whose densities are the solutions of \eqref{hydro1}-\eqref{hydro3}.
Since convergence in distribution to a deterministic variable implies convergence in probability, we get
\begin{align}\label{er}
\rho\pare{\ul\pi_\gamma\pare{\cdot}, \ul\pi^*\pare{\cdot}}\xrightarrow[\gamma\to 0]{}0
\end{align}
in $\BB P$-probability.
From \eqref{er} and the fact that $\ul\pi^*\pare{\cdot}\in C\pare{[0,T], \CAL M_1}^3$, as we have already proved in Lemma \eqref{nnb}, we can conclude our assertion.

\bigskip

{\bf Acknowledgments.}
It is a great pleasure to thank Errico Presutti for suggesting us the problem and for his continuous advising.
We also acknowledge (in alphabetical order) fruitful discussions with In\'es Armend\'ariz,  Anna De Masi, Pablo Ferrari, Claudio Landim,  Ellen Saada, Livio Triolo, and Maria Eul\'alia Vares.
The authors also acknowledge the hospitality of Laboratorie MAP5 at Universit\'e Paris Descartes.

\bibliographystyle{amsalpha}
\bibliography{biblio}

\end{document}